\documentclass[a4paper]{amsart}
\usepackage[utf8]{inputenc}
\usepackage[T1]{fontenc}
\usepackage[safeinputenc=true,style=alphabetic,firstinits,useprefix=true,doi=true,isbn=false,url=false]{biblatex}
\usepackage{hyperref}
\usepackage{amsmath,amsthm,amssymb}
\usepackage{esint}
\usepackage{mathtools}
\usepackage{tikz}
\DeclareMathOperator{\supp}{supp}
\newcommand{\size}{\mathbf{size}}
\newcommand{\parent}{\mathop{\mathrm{par}}}
\newcommand{\DI}{\mathcal{I}}
\newcommand{\Leaves}{\mathcal{L}}
\newcommand{\Z}{\mathbb{Z}}
\newcommand{\R}{\mathbb{R}}
\newcommand{\W}{\mathbb{W}}
\newcommand{\tr}{\mathrm{tr}}
\newcommand{\dif}{\mathrm{d}}
\newcommand{\h}{\mathrm{h}}
\newcommand{\p}{\mathrm{p}}
\def\<{\left\langle}
\def\>{\right\rangle}
\newcommand{\eset}{B}
\newcommand{\wplus}{\oplus}
\newcommand{\wtimes}{\circledast}

\numberwithin{equation}{section}
\theoremstyle{plain}

\newtheorem{theorem}[equation]{Theorem}

\newtheorem{proposition}[equation]{Proposition}

\theoremstyle{definition}
\newtheorem{definition}[equation]{Definition}

\title[Dyadic triangular Hilbert transform]{Dyadic triangular Hilbert transform of two general and one not too general function}

\author{Vjekoslav Kova\v{c}}
\address[VK]{Department of Mathematics, Faculty of Science, University of Zagreb, Bijeni\v{c}ka cesta 30, 10000 Zagreb, Croatia}
\email{vjekovac@math.hr}

\author{Christoph Thiele}
\author{Pavel Zorin-Kranich}
\address[CT, PZ]{Mathematical Institute, University of Bonn, Endenicher Allee 60, 53115 Bonn, Germany}
\email{thiele@math.uni-bonn.de}
\email{pzorin@math.uni-bonn.de}

\thanks{V.\ K.\ was partially supported by the Croatian Science Foundation under the project 3526.
C.\ T.\ was partially supported by the NSF grant DMS 1001535.}

\date{\today}

\begin{document}
\maketitle
\allowdisplaybreaks[3]

\begin{abstract}
The so-called triangular Hilbert transform is an elegant trilinear singular integral form which specializes to many well studied objects of harmonic analysis.
We investigate $L^p$ bounds for a dyadic model of this form in the particular case when one of the functions on which it acts is essentially one-dimensional.
This special case still implies dyadic analogues of boundedness of the Carleson maximal operator and of the uniform estimates for the one-dimensional bilinear Hilbert transform.
\end{abstract}

\section{Introduction}
\subsection{Motivation}
In this article we begin the study of a dyadic model of the ``triangular'' Hilbert transform.
In order to motivate its definition consider the family of trilinear forms (dual to two-dimensional bilinear Hilbert transforms):
\begin{equation}
\label{eq:THT-real-def1}
\Lambda_{\vec\beta_{0},\vec\beta_{1},\vec\beta_{2}}(F_{0},F_{1},F_{2})
:=
\iint_{\R^{2}} \mathrm{p.v.}\int_{\R} \prod_{i=0}^{2}F_i(\vec x-\vec\beta_{i}t) \frac{\dif t}{t} \dif \vec x,
\end{equation}
where $\vec\beta_{i}\in\R^{2}$ are distinct points.
When all three points $\vec\beta_{i}$ lie on the same line, these forms reduce to integrals of one-dimensional bilinear Hilbert transforms, and by the results of \cites{MR1491450,MR1689336} we have the $L^{p}$ bounds
\begin{equation}
\label{eq:triangular-hilbert-estimate}
|\Lambda_{\vec\beta_{0},\vec\beta_{1},\vec\beta_{2}}(F_{0},F_{1},F_{2})|
\lesssim_{\vec\beta_{0},\vec\beta_{1},\vec\beta_{2}} \prod_{i=0}^{2} \| F_{i} \|_{p_{i}}
\end{equation}
for all $1<p_{i}<\infty$ with
\begin{equation}
\label{eq:Lp-range}
\sum_{i=0}^{2} \alpha_{i} = 1,
\quad
\alpha_{i} = \frac{1}{p_{i}}.
\end{equation}
By scaling, the condition \eqref{eq:Lp-range} is necessary.

The case when all $\vec\beta_{i}$'s lie on the same line can be considered as a limiting case of the more general situation when the $\beta_{i}$'s are in general position.
Bounds of the form \eqref{eq:triangular-hilbert-estimate} in this general case would unify some of the central results in time-frequency analysis:
\begin{enumerate}
\item From the general case one can recover $L^{p}$ bounds for the Carleson maximal operator by making an appropriate choice of the functions $F_{i}$; see Appendix~\ref{appendix:real}.
\item From the general case one can also recover the uniform bounds for the one-dimensional bilinear Hilbert transform in \cite{MR2113017}. 
In fact, if the estimate \eqref{eq:triangular-hilbert-estimate} is true for any triple of $\vec\beta_{i}$'s in general position, then it is true for every such triple and the implied constant does not depend on the triple.
This calculation is carried out in Appendix~\ref{appendix:real}, where it is verified that the constant will be the same as in the corresponding estimate for the trilinear form
\begin{equation}
\label{eq:THT-real-def2}
\Lambda_{\Delta}(F_{0},F_{1},F_{2})
:=
\mathrm{p.v.} \iiint_{\R^{3}} F_{0}(x,y) F_{1}(y,z) F_{2}(z,x) \frac{\dif(x,y,z)}{x+y+z},
\end{equation}
which can be called the \emph{triangular Hilbert transform}.
\item By the method of rotations a hypothetical estimate for \eqref{eq:THT-real-def1} also implies $L^{p}$ estimates for the ``less singular'' bilinear singular integrals from \cite{MR2597511} uniformly over all choices of the ``direction matrices'', at least for odd two-dimensional kernels; see Appendix~\ref{appendix:real}.
\end{enumerate}

Unfortunately, the desired estimates for the triangular singular form \eqref{eq:THT-real-def2} still seem to be out of reach of the current techniques and, from what we have said, they are expected to be highly nontrivial.
In this paper we work in a dyadic model instead of the classical one and consider a particular case when one of the functions $F_0,F_1,F_2$ takes a special form.
This case still turns out to be general enough to imply (dyadic versions of) both the $L^{p}$ bounds for the Carleson operator and uniform bounds for the bilinear Hilbert transform.
In this sense our result has stronger one-dimensional consequences than the dyadic version of the argument from \cite{MR2597511} which appears in \cite{2012arXiv1210.0886D}: the latter does not contain the uniform bounds for the bilinear Hilbert transform.

\subsection{Notation}
Let us now introduce the dyadic model for the triangular Hilbert transform.
In this model the real line is replaced by the (Walsh) field $\W=\mathbb{F}_{2}((1/t))$ of one-sidedly infinite power series with coefficients in the two-element field $\mathbb{F}_{2}$.
The field $\W$ is traditionally identified with $[0,\infty)$ via the map $\sum_{k}a_{k}t^{k}\mapsto \sum_{k}a_{k}2^{-k}$, where $\mathbb{F}_{2}$ is identified with $\{0,1\}$.
This map is one-to-one on a conull set, and we normalize the Haar measure on $\mathbb{W}$ in such a way that this map becomes measure-preserving.
Under this identification the addition $\wplus$ and the multiplication $\wtimes$ on $\mathbb{W}$ correspond to addition and multiplication of binary numbers without carrying over digits.
We refer for instance to \cite[\S 1]{MR2692998} for more details.

The sets
\[
A_{k} := [0,2^{k}),
\quad k\in\Z
\]
then become additive subgroups and their cosets are simply \emph{dyadic intervals} of length $2^k$, the collection of which will be denoted by $\mathbf{I}_{k}$.
Some dyadic intervals (typically denoted by Latin letters, such as $I$) will be interpreted as time intervals and they will always be subsets of the unit interval $[0,1)$.
Other dyadic intervals will be interpreted as frequency intervals (typically denoted by Greek letters, such as $\omega$) and they will have integer endpoints.
For a dyadic interval $I$ we write $I^1$ for its left half and $I^{-1}$ for its right half.
The unique dyadic parent of $I$ will be denoted $\parent I$.
When we mention a \emph{dyadic square} we will always mean a dyadic square contained in $[0,1)^2$.

We work with real-valued functions, which is no restriction since all systemic functions under consideration, most notably the Haar functions, are real valued.
Let us then reserve the letter $i$ to denote an index $i\in \{0,1,2\}$.
It is convenient to regard $i$ as an element of $\Z/3\Z$ and interpret $i+1$ and $i-1$ correspondingly.
We shall also consider the set $\DI_{k}$ of all triples $\vec I = (I_0,I_1,I_2)$ of dyadic intervals contained in $[0,1)$ such that
\[
|I_0|=|I_1|=|I_2|=2^{k}, \quad 0\in I_0\wplus I_1 \wplus I_2
\]
and set $\DI = \cup_{k\leq 0} \DI_{k}$.
We write
\[
(I_0,I_1,I_2)\subset (J_0,J_1,J_2)
\]
if $I_{i}\subset J_{i}$ for $i=0,1,2$.

Any function $F$ on the unit square shall be interpreted as the integral operator
\[
(F\varphi)(x):=\int_0^1 F(x,y)\varphi(y)\, \dif y
\]
on $L^2([0,1))$, denoted by the same letter.
For any dyadic interval $I$ we normalize the \emph{Haar function} $\h_{I}$ in $L^\infty$, so that $\h_{I} = \sum_{j\in\{\pm 1\}} j 1_{I^{j}}$.
We shall also write $\h_I$ for the spatial multiplier operator acting on $L^2([0,1))$ and defined by
\[
(\h_I \varphi)(x):=\h_I(x)\varphi(x) .
\]

The \emph{dyadic triangular Hilbert transform} can be written as
\begin{equation}
\label{eq:THT-def}
\Lambda^{\epsilon}(F_0,F_1,F_2)
:=
\sum_{\vec I \in \DI} \epsilon_{\vec I}
|I_i|^{-1}\tr(\h_{I_i} F_{i-1}  \h_{I_{i+1}} F_i \h_{I_{i-1}} F_{i+1}),
\end{equation}
where $(\epsilon_{\vec I})_{\vec I\in\DI}$ is an arbitrary sequence of scalars bounded in magnitude by $1$ and $i\in\{0,1,2\}$ is a fixed index.
The expression does not depend on the specific choice of $i$ by cyclicity of the trace.
If the reader prefers an explicit integral representation, then \eqref{eq:THT-def} can easily be rewritten as
\begin{align}
& \Lambda^{\epsilon}(F_0,F_1,F_2) \label{eq:THT-def2} \\
& = \sum_{\vec I \in \DI} \epsilon_{\vec I}\, |I_0|^{-1}
\iiint_{\W^{3}} \h_{I_1}(x)F_0(x,y)\h_{I_2}(y)F_1(y,z)\h_{I_0}(z)F_2(z,x) \dif x \dif y \dif z, \nonumber
\end{align}
but we will continue to use the convenient ``trace-operator'' notation.
We note that \eqref{eq:THT-def2} is a perfect Calder\'on--Zygmund kernel analogue of \eqref{eq:THT-real-def2}; i.e.
\[
\sum_{\vec I \in \DI} \epsilon_{\vec I} |I_0|^{-1} \h_{I_1}(x)\h_{I_2}(y)\h_{I_0}(z)
\]
replaces $1/(x+y+z)$.
It is necessary to insert the coefficients $\epsilon_{\vec I}$, as otherwise the above kernel would telescope to the Dirac mass $\delta_0$ evaluated at $x+y+z$, and the form would become trivial.
Informally speaking, the Walsh model cannot distinguish between $\mathrm{p.v.}\frac{1}{t}$ and $\delta_0(t)$, so it becomes faithful only after breaking the form into scales.
We obtain the following strong type estimates (see Section~\ref{sec:wave-packets} for the definition of the character $e$).
\begin{theorem}
\label{thm:main}
Let $F_{0},F_{1},F_{2}$ be functions supported on $A_{0}^{2}$.
Suppose that either
\begin{equation}
\label{eq:F0-diagonal}
F_{0}(x_{1},x_{2})=f\big(x_{2}\wplus (a\wtimes x_{1})\big)
\text{ for all }x_{1},x_{2}\in A_{0}
\end{equation}
holds with some $a\in \W\setminus A_{0}$ and some measurable $f:\W\to\R$ or
\begin{equation}
\label{eq:F0-fiberwise-character}
F_{0}(x_{1},x_{2})=f(x_{2}) e(N_{x_{2}} \wtimes x_{1})
\text{ for all }x_{1},x_{2}\in A_{0}
\end{equation}
holds with some measurable $N : \W\to\W$ and $f:\W\to\R$.
Then
\begin{equation}
\label{eq:lp-estimate}
|\Lambda^{\epsilon}(F_0,F_1,F_2)|
\lesssim
\|F_{0}\|_{p_{0}} \|F_{1}\|_{p_{1}} \|F_{2}\|_{p_{2}}
\end{equation}
for any $1<p_{2}<\infty$ and $2<p_{0},p_{1}<\infty$ with \eqref{eq:Lp-range}.
The implied constant does not depend on $a$, $N$, or the scalars $|\epsilon_{\vec I}|\leq 1$ with $\epsilon_{\vec I}=0$ whenever some $I_{i} \not\subseteq A_{0}$.
In case \eqref{eq:F0-fiberwise-character} we can relax the restriction on $p_{0}$ to $1<p_{0}<\infty$.
In case \eqref{eq:F0-diagonal}, $a\in A_{1}\setminus A_{0}$, we can relax the restrictions on both $p_{0}$ and $p_{1}$ to $1<p_{0},p_{1}<\infty$.
\end{theorem}
The range of exponents $\alpha_{i}$ covered by Theorem~\ref{thm:main-restricted-type} is the union of the triangles $c,b_{0},b_{1},b_{2}$ in Figure~\ref{fig:exponents}.
Since the conditions \eqref{eq:F0-diagonal} and \eqref{eq:F0-fiberwise-character} (with $a$ and $N$ fixed) describe subspaces of $L^{p_{0}}(\W^{2})$ that are themselves $L^{p_{0}}$ spaces, Theorem~\ref{thm:main} follows by real interpolation from (generalized) restricted weak type estimates.
Note that the range in which such estimates hold is clearly convex in Figure~\ref{fig:exponents}

The local $L^{2}$ case (triangle $c$) is covered by Proposition~\ref{prop:restricted-type}.
In this case the localization $I_{i} \subseteq A_{0}$ can be removed by the Loomis--Whitney inequality.
Triangle $d_{12}$ is covered by Theorem~\ref{thm:main-restricted-type}; this gives the lower half of the solid hexagon in Figure~\ref{fig:exponents}.
Triangle $d_{10}$ in cases \eqref{eq:F0-fiberwise-character} and \eqref{eq:F0-diagonal}, $a\in A_{1}\setminus A_{0}$, is covered by Theorem~\ref{thm:main-restricted-type:F0}; together with the previous result this gives the full solid hexagon in Figure~\ref{fig:exponents}.
Finally, the case \eqref{eq:F0-diagonal}, $a\in A_{1}\setminus A_{0}$, is symmetric in indices $0,2$; in this case we obtain estimates in the dashed extension of the solid hexagon in Figure~\ref{fig:exponents}.

The cases \eqref{eq:F0-diagonal} and \eqref{eq:F0-fiberwise-character} are treated in a unified way and cover all types of functions used to recover algebraically defined dyadic models for the Carleson operator and uniform estimates for the bilinear Hilbert transform, see Appendix~\ref{appendix:dyadic}.
However, note that already one type, namely the case \eqref{eq:F0-diagonal}, $a\in A_{1}\setminus A_{0}$, suffices to recover the bounds for both these operators.
In particular we recover the full range of exponents for which uniform estimates for the (dyadic) bilinear Hilbert transform are known (triangles $a_{1},a_{2}$ have been treated in \cite{MR2320411} and triangles $d_{12},d_{21}$ in \cite{MR2997005}).
This range seems to be the best possible, because for $\alpha_{\pm 1}\leq -1/2$ there are indications that even non-uniform bounds for the bilinear Hilbert transform fail, whereas for $\alpha_{0}\leq 0$ the bounds fail in the limiting case of the $1$-linear Hilbert transform.

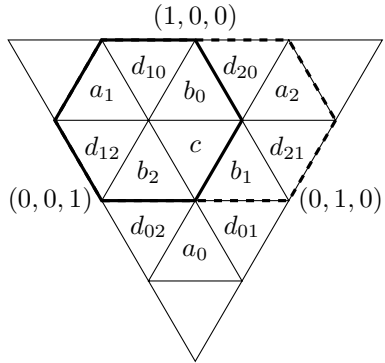
\begin{figure}
\begin{center}
\begin{tikzpicture}[x=70pt,y=70pt,cm={sin(30),-cos(30), -sin(30) ,-cos(30) ,(0,0)}]
\draw
(-1,1) -- (1,1) -- (1,-1) -- cycle
(1,-0.5) -- (0.5,-0.5) -- (0.5,1) -- (1,0.5) -- (-0.5,0.5) -- (-0.5,1) -- cycle
(1,0) node[right] {$(0,1,0)$} -- (0,0) node[above] {$(1,0,0)$} -- (0,1) node[left] {$(0,0,1)$} -- cycle
(0.33,0.33) node{$c$}
(0.16,0.16) node{$b_{0}$}
(0.16,0.66) node{$b_{2}$}
(0.66,0.16) node{$b_{1}$}
(0.33,0.83) node{$d_{02}$}
(-0.16,0.83) node{$d_{12}$}
(0.83,0.33) node{$d_{01}$}
(0.83,-0.16) node{$d_{21}$}
(0.33,-0.16) node{$d_{20}$}
(-0.16,0.33) node{$d_{10}$}
(0.66,0.66) node{$a_{0}$}
(0.66,-0.33) node{$a_{2}$}
(-0.33,0.66) node{$a_{1}$}
;
\draw[very thick] (0.5,0) -- (0.5,0.5) -- (0,1) -- (-0.5,1) -- (-0.5,0.5) -- (0,0) -- cycle;
\draw[very thick,dashed] (1,-0.5) -- (1,0) -- (0,1) -- (-0.5,1) -- (-0.5,0.5) -- (0.5,-0.5) -- cycle;
\end{tikzpicture}
\end{center}
\caption{Ranges of exponents in coordinates $(\alpha_{0},\alpha_{1},\alpha_{2})=(1/p_{0},1/p_{1},1/p_{2})$.}
\label{fig:exponents}
\end{figure}

\section{Tile decomposition}
In this section we describe a time-frequency decomposition for the form \eqref{eq:THT-def} that is well adapted both to diagonal functions \eqref{eq:F0-diagonal} and to fiberwise characters \eqref{eq:F0-fiberwise-character}.
While the decomposition of the \emph{form} is the same in both cases, the time-frequency projections of (one of) the \emph{functions} differ.
However, in both cases the time-frequency projections satisfy the same localization and scale compatibility properties, summarized in Definition~\ref{tile-proj}.
The proof of the local $L^{2}$ bounds uses only these properties and a single tree estimate.
We will have to come back to the definition of time-frequency projections in the multi-frequency Calder\'on--Zygmund decomposition in Section~\ref{sec:mfcz}.

\subsection{Wave packets}
\label{sec:wave-packets}
The characters on the Walsh field $\W$ are the \emph{Walsh functions}
\[
w_{N}(x) := e(N \wtimes x),
\]
where $N\in\W$ and $e\colon\W\to\R$ is simply the periodization of $\h_{[0,1)}$.
Their particular cases are the \emph{Rademacher functions} $r_k:=w_{2^{-k}}$, $k\in\Z$.
The \emph{Walsh wave packet} associated with a dyadic rectangle $I\times\omega$ of area $1$ is
\[
w_{I\times\omega}(x) := |I|^{-1/2} 1_{I}(x) e(l(\omega) \wtimes x),
\]
where $l(\omega)$ is the left endpoint of $\omega$.
This definition satisfies the usual recursive relations
\[
w_{P_{\mathrm{up}}} = (w_{P_{\mathrm{left}}} - w_{P_{\mathrm{right}}})/\sqrt{2},
\quad
w_{P_{\mathrm{down}}} = (w_{P_{\mathrm{left}}} + w_{P_{\mathrm{right}}})/\sqrt{2}
\]
on every dyadic rectangle $P$ of area $2$ and therefore coincides with the usual definition; see \cite[\S 1]{MR2692998}.

\subsection{Tile decomposition}
Our time-frequency analysis is $1\frac12$-dimensional in the sense of \cite{MR2597511}.
We define \emph{tiles} as dyadic boxes
\[
\p = I_{\p,0}\times I_{\p,2} \times \omega_{\p,1},
\quad\text{where}\quad
|I_{\p,0}| = |I_{\p,2}| = |\omega_{\p,1}|^{-1}.
\]
A \emph{bitile} is then any dyadic box of the form
\[
P=I_{P,0}\times I_{P,2} \times \omega_{P,1},
\quad\text{where}\quad
|I_{P,0}| = |I_{P,2}| = 2|\omega_{P,1}|^{-1}.
\]
We will omit the subscripts $\p,P$ if no confusion seems possible.
For notational convenience we will throughout write $I_{1}=I_{0}\wplus I_{2}$.

Dyadic boxes are partially ordered by
\[
P \leq P' :\iff I_{i} \subseteq I_{i}',\ \omega_{i} \supseteq \omega_{i}'.
\]
Writing one of the Haar functions in \eqref{eq:THT-def} as a difference of two characteristic functions we arrive at
\[
\Lambda^{\epsilon}(F_{0},F_{1},F_{2})
=
\sum_{\vec I} \epsilon_{\vec I} \sum_{j \in \{\pm 1\}} j |I_{1}|^{-1}
\tr(1_{I_{1}^{j}}1_{I_{1}^{j}}F_{0} \h_{I_{2}} F_{1} \h_{I_{0}} F_{2}),
\]
where $1_I$ denotes, along with the characteristic function of the interval $I$, also the projection operator
\[
(1_I \varphi)(x)=1_I(x)\varphi(x).
\]
Inserting identity operators (expanded in the Walsh basis) between characteristic functions we obtain
\[
\sum_{\vec I} \epsilon_{\vec I} \sum_{j \in \{\pm 1\}} j
\sum_{\omega_{1} : |\omega_{1}| = 2|I_{1}|^{-1}} 2|I_{1}|^{-2}\\
\tr\big(
1_{I_{1}^{j}}(w_{l_{1}}\otimes w_{l_{1}})1_{I_{1}^{j}} F_{0}
\h_{I_{2}} F_{1}
\h_{I_{0}} F_{2}\big).
\]
Changing the order of summation we obtain
\[
\Lambda^{\epsilon}(F_{0},F_{1},F_{2})
=
\sum_{P \text{ bitile}} \epsilon_{\vec I_{P}}
\Lambda_{P}(F_{0},F_{1},F_{2}),
\]
where
\[
\Lambda_{\vec I\times \vec\omega}(F_{0},F_{1},F_{2})
:=
\sum_{j \in \{\pm 1\}} j 2|I_{1}|^{-2}
\tr(1_{I_{1}^{j}}(w_{l}\otimes w_{l})1_{I_{1}^{j}}F_{0} \h_{I_{2}} F_{1} \h_{I_{0}} F_{2}).
\]
Note that each $l$ can be replaced by any frequency from $\omega$ since this only multiplies the corresponding character by a constant on each of the intervals $I_{i}^{j}$.

\subsection{Time-frequency projections}
We begin by collecting desirable properties of time-frequency projections.
\begin{definition}
\label{tile-proj}
We call orthogonal projections $\Pi^{(i)}_{\p}$, acting on $L^{2}(x_{i-1},x_{i+1})$ and indexed by tiles $\p$, \emph{time-frequency projections} if they satisfy the following conditions.
\begin{enumerate}
\item\label{tile-proj:orth} (Orthogonality) The projections $\Pi^{(i)}_{\p}$ corresponding to disjoint tiles are orthogonal.
\item\label{tile-proj:compatible} (Scale compatibility) Bitile projections $\Pi^{(i)}_{P}$ are well-defined (there are two ways to write a bitile as a disjoint union of tiles, and the corresponding sums of tile projections are equal).
\item\label{tile-proj:support} (Support) $\supp \Pi_{\p}^{(i)}F_{i} \subset I_{i-1}\times I_{i+1}$.
\end{enumerate}
\end{definition}
A collection of bitiles $\mathbf{P}$ is called \emph{convex} if $P,P''\in\mathbf{P}$, $P\leq P'\leq P''$ implies $P'\in\mathbf{P}$.
The union of any finite convex collection of bitiles $\mathbf{P}$ can be written as the union of a collection of disjoint tiles $\mathbf{p}$ (this is proved by induction on the number of bitiles, cf.\ \cite[Lemma 1.7]{MR2692998}).
Given time-frequency projections, this allows us to consider the projections
\[
\Pi_{\mathbf{P}}^{(i)}F_{i} := \sum_{\p\in\mathbf{p}} \Pi_{\p}^{(i)}F_{i}.
\]
The property \eqref{tile-proj:compatible} implies that these projections do not depend on the choice of $\mathbf{p}$, cf.\ \cite[Corollary 1.9]{MR2692998}.

\begin{definition}
We call time-frequency projections \emph{adapted} to $F_{0}$ if for every choice of $F_{1},F_{2}$, every bitile $P$, and any convex collection of bitiles $\mathbf{P}\ni P$ we have
\begin{equation}
\label{eq:Lambda-tile-proj}
\Lambda_{P}(F_{0},F_{1},F_{2}) = \Lambda_{P}(\Pi_{\mathbf{P}}^{(0)}F_{0},\Pi_{\mathbf{P}}^{(1)}F_{1},\Pi_{\mathbf{P}}^{(2)}F_{2}).
\end{equation}
\end{definition}

The existence of adapted time-frequency projections suffices to establish restricted type bounds on the dyadic triangular Hilbert transform in the local $L^{2}$ range.
\begin{proposition}
\label{prop:restricted-type}
Let $E_{i} \subset A_{0}^{2}$, $i\in\{0,1,2\}$, be measurable sets and $|F_{i}| \leq 1_{E_{i}}$ be functions for which there exist time-frequency projections adapted to $F_{0}$.
Then
\[
|\Lambda^{\epsilon}(F_{0},F_{1},F_{2})|
\lesssim
a_{1}^{1/2} a_{2}^{1/2} (1+\log\frac{a_{0}}{a_{1}}),
\]
where $a_{i}=|E_{\sigma(i)}|$ is a decreasing rearrangement, that is, $\sigma$ is a permutation of $\{0,1,2\}$ and $a_{0}\geq a_{1}\geq a_{2}$.
The implied constant is independent of the choices of the scalars $|\epsilon_{\vec I}|\leq 1$.
\end{proposition}

We finish this section with the construction of time-frequency projections adapted to \eqref{eq:F0-diagonal} and \eqref{eq:F0-fiberwise-character}.
For indices $0$ and $2$ we use the projections
\begin{equation}
\label{eq:Pi2}
\Pi^{(2)}_{\p} F_{2}(x_{0},x_{1}) := 1_{I_{0}}(x_{0}) \< F_{2}(x_{0},\cdot), w_{I_{1}\times \omega_{1}} \> w_{I_{1}\times \omega_{1}}(x_{1})
\end{equation}
and
\begin{equation}
\label{eq:Pi0}
\Pi^{(0)}_{\p} F_{0}(x_{2},x_{1}) := 1_{I_{2}}(x_{2}) \< F_{0}(x_{2},\cdot), w_{I_{1}\times \omega_{1}} \> w_{I_{1}\times \omega_{1}}(x_{1}).
\end{equation}
The structural information given by \eqref{eq:F0-diagonal} and \eqref{eq:F0-fiberwise-character} is encoded in the projections $\Pi^{(1)}$.

\subsubsection{One-dimensional functions}
\label{sec:1dimfct}
Suppose \eqref{eq:F0-diagonal}.
Then we have
\[
\Pi^{(0)}_{\p}F_{0}(x_{1},x_{2}) = 1_{I_{1}}(x_{1}) (\Pi_{I_{2}\times a\wtimes\omega_{1}}F_{0}(\cdot,x_{1}))(x_{2}),
\]
where the projection on the right-hand side is a one-dimensional time-frequency projection (as defined e.g.\ in \cite{MR2997005}) with a possibly multidimensional range.
In this case we define
\[
\Pi^{(1)}_{P}F_{1}(x_{2},x_{0}) := 1_{I_{0}}(x_{0}) (\Pi_{I_{2}\times a\wtimes\omega_{1}}F_{1}(\cdot,x_{0}))(x_{2}).
\]

\subsubsection{Fiberwise characters}
\label{sec:fiberwisechar}
Suppose \eqref{eq:F0-fiberwise-character}.
Then we have
\[
\Pi^{(0)}_{\p}F_{0}(x_{1},x_{2}) = 1_{I_{1}}(x_{1}) 1_{I_{2}}(x_{2}) 1_{\omega_{1}}(N_{x_{2}}) F_{0}(x_{1},x_{2}).
\]
In this case we define
\[
\Pi^{(1)}_{\p}F_{1}(x_{2},x_{0}) := 1_{I_{0}}(x_{0}) 1_{I_{2}}(x_{2}) 1_{\omega_{1}}(N_{x_{2}}) F_{1}(x_{2},x_{0}).
\]
The projections $\Pi^{(1)}$ constructed above satisfy \eqref{eq:Lambda-tile-proj} only for bitiles with $I_{i}\subseteq A_{0}$, which explains the truncation in Theorem~\ref{thm:main}.

\section{Single tree estimate}
\label{sec:single-tree}
A \emph{tree} $T$ is a convex set of bitiles that contains a maximal element
\[
P_{T} = \vec I_{T} \times \vec \omega_{T} = I_{T,0} \times I_{T,2} \times \omega_{T,1}.
\]
Equivalently, a tree can be described by a top frequency $\xi_{T,1}$ and a convex collection of space boxes $\DI_{T}$.
The corresponding tree $T$ then consists of all bitiles $P = \vec I\times\vec\omega$ with $\vec I\in \DI_{T}$ and $\xi_{T,1}\in\omega_{1}$.

For a convex collection $\mathbf{P}$ of bitiles define
\begin{equation}
\label{eq:size}
\size^{(i)}(\mathbf{P},F_{i})
:=
\sup_{T\subset \mathbf{P} \text{ tree}} |\vec I_{T}|^{-1/2} \|\Pi_{T}^{(i)}F_{i}\|_{2}.
\end{equation}
For a collection $\mathbf{P}$ of bitiles write
\[
\Lambda^{\epsilon}_{\mathbf{P}}(F_{0},F_{1},F_{2})
:=
\sum_{P\in\mathbf{P}} \epsilon_{\vec I_{P}}
\Lambda_{P}(F_{0},F_{1},F_{2}).
\]
The objective of this section is to show that Definition~\ref{tile-proj} implies
\begin{equation}
\label{eq:single-tree-estimate:symmetric}
|\Lambda^{\epsilon}_{T}(F_{0},F_{1},F_{2})|
\lesssim
|\vec I_{T}| \prod_{i=0}^{2} \size^{(i)}(T,F_{i}),
\end{equation}
where $T$ is a tree and the implied constant is absolute.
It follows from Definition~\ref{tile-proj} that
\[
|\vec I_{P}|^{-1/2} \|\Pi^{(i)}_{P} F_{i}\|_{L^{2}(I_{i-1,P}\times I_{i+1,P})} \lesssim \size^{(i)}(T,F_{i})
\quad
\text{ for all } P\in T.
\]
Thus in view of \eqref{eq:Lambda-tile-proj} it suffices to show
\begin{equation}
\label{eq:single-tree-estimate:symmetric}
|\Lambda^{\epsilon}_{T}(F_{0},F_{1},F_{2})|
\lesssim
|\vec I_{T}| \prod_{i=0}^{2} \sup_{\vec I\in \DI_{T} \cup \Leaves_{T}} |\vec I|^{-1/2} \|F_{i}\|_{L^{2}(I_{i-1}\times I_{i+1})},
\end{equation}
where $\Leaves_{T}$ denotes the collection of leaves of a tree, that is, maximal elements of $\DI$ contained in a member of $T$ that are not themselves members of $\DI_{T}$.
By modulation we may assume $\xi_{T,1}=0$.
The tree operator can be written
\[
\sum_{\vec I\in \DI_{T}} \epsilon_{\vec I} |I_{1}|^{-2}
(\tr((1\otimes 1)1_{I_{1}} F_{0} \h_{I_{2}} F_{1} \h_{I_{0}} F_{2} \h_{I_{1}})
+
\tr((1\otimes 1)\h_{I_{1}} F_{0} \h_{I_{2}} F_{1} \h_{I_{0}} F_{2} 1_{I_{1}})).
\]
The two summands are symmetric (under permuting the indices $0$ and $2$) and we consider only the first of them.
With the convention that the domain of integration is $x_{i},y_{i}\in I_{i}$ and the dyadic intervals have size $|I_{i}|=2^{k}$ we have
\begin{multline*}
\tr((1\otimes 1)1_{I_{1}} F_{0} \h_{I_{2}} F_{1} \h_{I_{0}} F_{2} \h_{I_{1}})\\
=
\int F_{0}(x_{1},x_{2}) r_{k}(x_{2}) F_{1}(x_{2},x_{0}) r_{k}(x_{0}) F_{2}(x_{0},y_{1}) r_{k}(y_{1})
\dif x_{1} \dif x_{2} \dif x_{0} \dif y_{1}.
\end{multline*}
The change of variables $x_{1}=x_{2}+y_{0}$, $y_{1}=x_{0}+y_{2}$ gives
\begin{align*}
&\int F_{0}(x_{2}+y_{0},x_{2}) r_{k}(x_{2}) F_{1}(x_{2},x_{0}) r_{k}(x_{0}) F_{2}(x_{0},x_{0}+y_{2}) r_{k}(x_{0}+y_{2})
\dif y_{0} \dif x_{2} \dif x_{0} \dif y_{2}\\
&=\int \tilde F_{0}(y_{0},x_{2}) r_{k}(x_{2}) F_{1}(x_{2},x_{0}) \tilde F_{2}(x_{0},y_{2}) r_{k}(y_{2})
\dif y_{0} \dif x_{2} \dif x_{0} \dif y_{2},
\end{align*}
where $\tilde F_{0}(y_{0},x_{2}):=F_{0}(x_{2}+y_{0},x_{2})$ and $\tilde F_{2}(x_{0},y_{2}) := F_{2}(x_{0},x_{0}+y_{2})$.

Thus the first half of the tree operator can be written as a single tree operator from \cite[\textsection 3]{MR2990138} with square-dependent coefficients.
The first step in the proof of \cite[Proposition 4]{MR2990138} is an application of the Cauchy--Schwarz inequality in the sum over squares, so it still works in our situation.
This, together with \cite[(2.2)]{MR2990138}, gives the required estimate.

\section{Tree selection and local $L^{2}$ bounds}
\subsection{The tree selection algorithm}
We organize bitiles into trees closely following the argument in \cite[Lemma 2.2]{MR2997005}.
Here and later we use coordinate projections $\pi_{(i)}: \W^{3}\to\W^{2}, (x_{i-1},x_{i},x_{i+1}) \mapsto (x_{i-1},x_{i+1})$.
\begin{proposition}
\label{prop:tree-selection}
Let $n\in\Z$, $i\in\{0,1,2\}$, a function $F_{i}$, and a system of (not necessarily adapted) time-frequency projections $\Pi^{(i)}$ be given.
Then every finite convex collection of bitiles $\mathbf{P}$ can be partitioned into a convex collection of bitiles $\mathbf{P}'$ with
\[
\size_{i}(\mathbf{P}',F_{i}) \leq 2^{-n}
\]
and a further convex collection of bitiles that is the disjoint union of a collection of convex trees $\mathbf{T}$ with
\begin{equation}
\label{eq:tree-selection:tree-counting-bound}
\sum_{T\in\mathbf{T}, \vec I_{T} \subset \vec J} |\vec I_{T}|
\leq
9 \cdot 2^{2n} \|1_{\pi_{(i)}\vec J} F_{i}\|_{2}^{2},
\quad \vec J\in\DI.
\end{equation}
\end{proposition}
The latter bound includes both an $L^{1}$ estimate (taking $\vec J$ large enough to contain all time intervals in $\mathbf{P}$) and a $\mathrm{BMO}$ estimate (noting $\|1_{\pi_{(i)}\vec J} F_{i}\|_{2}^{2} \leq |\vec J| \|F_{i}\|_{\infty}^{2}$) for the counting function $\sum_{T\in\mathbf{T}} 1_{\vec I_{T}}$.
\begin{proof}
We will remove three collections of trees, each of which satisfies \eqref{eq:tree-selection:tree-counting-bound} with a smaller constant.
At each step we remove a tree that is also a down-set, thus ensuring that both the remaining collection $\mathbf{P}'$ and the collection of all removed tiles are convex.

Replacing $F_{i}$ by $2^{n}F_{i}$ we may assume $n=0$.
We write every bitile $P$ as $P^{+1}\cup P^{-1}$, where the tiles $P^{j}$, $j=\pm 1$, are given by $\vec I_{P} \times \omega_{P,1}^{j}$.

For a tree $T$ write
\[
T_{j} := \{ P\in T : P^{j} \leq P_{T}\},
\quad
j=\pm 1.
\]
Then
\[
\Pi_{T}^{(i)} F_{i}
=
\Pi_{P_{T}}^{(i)} F_{i}
+
\sum_{j=\pm 1}\sum_{P\in T_{j}} \Pi_{P^{-j}}^{(i)} F_{i},
\]
and this sum is orthogonal by Definition~\ref{tile-proj}~\ref{tile-proj:orth}.

Let $\{P_{1},\dots,P_{n}\}$ be the collection of maximal bitiles in $\mathbf{P}$ that satisfy
\[
\| \Pi_{P_{k}}^{(i)} F_{i} \|_{2}^{2} > 3^{-1} | \vec I_{P_{k}} |.
\]
These bitiles are necessarily pairwise disjoint, so we have
\[
\sum_{k : \vec I_{P_{k}} \subset \vec J} | \vec I_{P_{k}} |
<
3 \sum_{k : \vec I_{P_{k}} \subset \vec J} \| \Pi_{P_{k}}^{(i)} F_{i} \|_{2}^{2}
\leq
3 \| 1_{\pi_{(i)}\vec J} F_{i} \|_{2}^{2}
\]
for every $\vec J\in\DI$, where the last inequality follows from parts \ref{tile-proj:orth} and \ref{tile-proj:support} of Definition~\ref{tile-proj}.
Thus, removing the bitiles $P\leq P_{k}$ from $\mathbf{P}$, we may assume
\[
\| \Pi_{P}^{(i)} F_{i} \|_{2}^{2} \leq 3^{-1} | \vec I_{P} |,
\quad
P\in\mathbf{P}.
\]
The next step will be done twice, for $j=\pm 1$.
In each case we remove a collection of trees $\mathbf{T}_{j}$ such that for every remaining tree $T$ we have
\begin{equation}
\label{eq:tree-selection:up}
\sum_{P\in T_{j}} \| \Pi_{P^{-j}}^{(i)} F_{i} \|^{2} \leq 3^{-1} |\vec I_{T}|^{2}.
\end{equation}
The collection $\mathbf{T}_{j}=\{T_{1},T_{2},\dots\}$ is selected iteratively.
Suppose that $T_{1},\dots,T_{k}$ have been selected and suppose that \eqref{eq:tree-selection:up} is violated for some remaining tree $T \subset \mathbf{P}\setminus T_{1}\cup\dots\cup T_{k}$.
Choose such tree for which either the left endpoint of $\omega_{T,1}$ is minimal (if $j=-1$) or the right endpoint is maximal (for $j=+1$) and let $T_{k+1}\subset\mathbf{P}$ be the down-set spanned by the chosen tree.

We claim that the tiles of the form $P_{m}^{-j}$, $P_{m}\in (T_{m})_{j}$, are pairwise disjoint.
This is clear within each tree, so assume for contradiction $P_{k}^{-j} < P_{l}^{-j}$, $k\neq l$.
In particular, we have $P_{k} < P_{l}$, and this implies $k<l$, since otherwise $P_{k}$ should have been included in $T_{l}$.
On the other hand, $\omega_{P_{k},1}^{-j} \supsetneq \omega_{P_{l},1}^{-j}$ implies $\omega_{P_{k},1}^{-j} \supseteq \omega_{P_{l},1} \supsetneq \omega_{P_{l},1}^{j} \supseteq \omega_{T_{l},1}$, whereas $\omega_{T_{k},1} \subseteq \omega_{P_{k},1}^{j}$.
Thus $\omega_{T_{k},1}$ is either to the right (if $j=-1$) or to the left (if $j=+1$) from $\omega_{T_{l},1}$, in both cases contradicting the choice of $T_{k}$.

Violation of \eqref{eq:tree-selection:up} for $T_{k}\in\mathbf{T}_{j}$ and parts \ref{tile-proj:orth} and \ref{tile-proj:support} of Definition~\ref{tile-proj} give
\[
\sum_{k : \vec I_{T_{k}} \subset \vec J} |\vec I_{T_{k}}|
<
\sum_{k : \vec I_{T_{k}} \subset \vec J} 3 \sum_{P\in (T_{k})_{j}} \| \Pi_{P^{-j}}^{(i)} F_{i} \|_{2}^{2}
\leq
3 \| 1_{\pi_{(i)}\vec J} F_{i} \|_{2}^{2},
\]
as required.
For each remaining tree we will have
\[
\| \Pi_{P_{T}}^{(i)} F_{i} \|^{2}
+
\sum_{j=\pm 1}\sum_{P\in T_{j}} \| \Pi_{P^{-j}}^{(i)} F_{i} \|^{2}
\leq
(3^{-1}+3^{-1}+3^{-1}) |\vec I_{T}|,
\]
and this gives the required estimate for $\size_{i}(T,F_{i})$.
\end{proof}

\subsection{Local $L^{2}$ bounds (triangle $c$)}
\begin{proof}[Proof of Proposition~\ref{prop:restricted-type}]
By the Loomis--Whitney inequality we may assume $\epsilon_{\vec I}=0$ whenever $|I_{i}|>2^{k}$.

Normalizing $\tilde F_{i} = F_{i}/|E_{i}|^{1/2}$ we have to show
\begin{equation}
\label{eq:restricted-type-rescaled}
|\Lambda_{\mathbf{P}}^{\epsilon}(\tilde F_{0},\tilde F_{1},\tilde F_{2})|
\lesssim
a_{0}^{-1/2} (1+\log\frac{a_{0}}{a_{1}})
\end{equation}
with a constant independent of the (finite) convex collection of bitiles $\mathbf{P}$.
We have $\size^{(i)}(\tilde F_{i}) \leq \|\tilde F_{i}\|_{\infty} \leq |E_{i}|^{-1/2} = a_{\sigma^{-1}(i)}^{-1/2}$ and $\|\tilde F_{i}\|_{2}\leq 1$.
Fix integers $n_{i}$ such that $2^{n_{i}-1} < a_{i}^{-1/2} \leq 2^{n_{i}}$; note that in particular $n_{0}\leq n_{1}\leq n_{2}$.
Running the tree selection algorithm (Proposition~\ref{prop:tree-selection}) iteratively at each scale $n\leq n_{2}$ for each $i\in\{0,1,2\}$ we obtain collections of trees $\mathbf{T}_{n}$ with
\[
\sum_{T\in\mathbf{T}_{n}} |I_{T,i}|^{2} \lesssim 2^{-2n}
\]
and
\[
\size^{(i)}(T,\tilde F_{i}) \leq \min (2^{n}, 2^{n_{\sigma^{-1}(i)}}),
\quad T\in \mathbf{T}_{n}.
\]
Summing the single tree estimate \eqref{eq:single-tree-estimate:symmetric} over all trees we obtain
\[
|\Lambda^{\epsilon}(\tilde F_{0},\tilde F_{1},\tilde F_{2})|
\lesssim
\sum_{n\leq n_{2}} 2^{-2n}
\prod_{i=0}^{2} \min (2^{n}, 2^{n_{i}}).
\]
The sum over $n$ is an increasing geometric series for $n<n_{0}$ and a decreasing geometric series for $n>n_{1}$.
In particular, the sum is dominated by the terms $n_{0}\leq n\leq n_{1}$, that is, we have the estimate
\[
2^{n_{0}}(1+n_{1}-n_{0})
\lesssim
a_{0}^{-1/2} (1+\log\frac{a_{0}}{a_{1}})
\]
as required.
\end{proof}

\section{Fiberwise multi-frequency Calder\'on--Zygmund decomposition and an extended range of exponents}
\label{sec:mfcz}
In order to extend the range of exponents in our main result we perform a fiberwise multi-frequency Calder\'on--Zygmund decomposition.
Here, in contrast to the local $L^{2}$ range, we have to use the special form of the time-frequency projections $\Pi^{(0)}$ and $\Pi^{(2)}$.

Our decomposition unites the main features of the one-dimensional multi-frequency Calder\'on--Zygmund decomposition in \cite{MR2997005} and the fiberwise single-frequency Calder\'on--Zygmund decomposition in \cites{MR3161332,MR2990138}.
A useful simplification with respect to \cite{MR2997005} is that we do not attempt to control the size of the good function, this corresponds to the observation that the argument on page 1709 of \cite{MR2997005} works directly for $a$ in place of $a_{m}$.

\subsection{Triangles $b_{2}$ and $d_{12}$}
\begin{theorem}
\label{thm:main-restricted-type}
Let $0< \alpha_{0} \leq 1/2 \leq \alpha_{2} < 1$ and $-1/2<\alpha_{1}<1/2$ satisfy \eqref{eq:Lp-range}.
Then for any measurable sets $E_{i} \subset A_{0}^{2}$, $i\in\{0,1,2\}$ there exists a major subset $E_{1}'\subset E_{1}$ (which can be taken equal to $E_{1}$ if $\alpha_{1}>0$)
such that for any dyadic test functions $|F_{i}| \leq 1_{E_{i}}$, $|F_{1}| \leq 1_{E_{1}'}$ with \eqref{eq:F0-diagonal} or \eqref{eq:F0-fiberwise-character} we have
\[
|\Lambda^{\epsilon}(F_0,F_1,F_2)|
\lesssim_{\alpha_{0},\alpha_{1},\alpha_{2}}
\prod_{i=0}^{2} |E_{i}|^{\alpha_{i}},
\]
where the implied constant is independent of the choices of the scalars $|\epsilon_{\vec I}|\leq 1$ with $\epsilon_{\vec I}=0$ whenever $I_{i} \not\subset A_{0}$.
\end{theorem}
\begin{proof}
The required estimate is invariant under rescaling by powers of $2$, so we may normalize $|E_{1}| \approx 1$.
The localization changes to $E_{i} \subset A_{k}^{2}$ for some $k\in\Z$, but all previous results still apply by scale invariance.
In the case $|E_{2}| \gtrsim |E_{1}|$ the estimate with $E_{1}'=E_{1}$ follows from the local $L^{2}$ case $0<\alpha_{0},\alpha_{1},\alpha_{2}\leq 1/2$, which is given by Proposition~\ref{prop:restricted-type}.
Thus we may assume $|E_{2}| < 2^{-20}$.

Define the exceptional sets
\[
\eset_{0} := \{M_{p_{0}}(|E_{0}|^{-1/p_{0}} 1_{E_{0}}) > 2^{10}\}
\]
and
\[
\eset_{2} := \{\tilde M_{p_{2}}(|E_{2}|^{-1/p_{2}} 1_{E_{2}}) > 2^{10}\},
\]
where $\tilde M_{p_{2}}$ is the directional maximal function (in the direction $x_{1}$).
The set
\[
\eset_{1}:=\pi_{(1)}((\pi_{(0)}^{-1}\eset_{0}\cup \pi_{(2)}^{-1}\eset_{2})\cap\Delta),
\quad
\Delta := \{x_{0}\wplus x_{1}\wplus x_{2} = 0\} \subset \W^{3},
\]
has measure $<1/2$ by the Hardy--Littlewood maximal inequality.
Consider the major subset $E_{1}':=E_{1} \setminus \eset_{1}$.

Define normalized functions
\[
\tilde F_{i} := |E_{i}|^{-1/p_{i}} F_{i}.
\]
By construction of the major subset only the bitiles $P$ with
\[
\pi_{(1)}\vec I_{P}\not\subset \eset_{1}
\]
contribute to the trilinear form $\Lambda$, so consider a finite convex collection $\mathbf{P}$ of such bitiles.
Since the $M_{p_{0}}$ maximal function dominates the $M_{2}$ maximal function pointwise and by Definition~\ref{tile-proj}~\ref{tile-proj:support} we have
\[
\size^{(0)}(\mathbf{P},\tilde F_{0}) \lesssim 1,
\quad
\size^{(1)}(\mathbf{P},\tilde F_{1}) \lesssim 1.
\]
By the tree selection algorithm in Proposition~\ref{prop:tree-selection} we partition $\mathbf{P}$ into a sequence of pairwise disjoint convex unions of pairwise disjoint trees $\mathbf{P}_{k} = \cup_{T\in\mathbf{T}_{k}} T$ and a remainder set with zero contribution to $\Lambda$ in such a way that
\[
\size^{(0)}(\mathbf{P}_{k},\tilde F_{0}) \lesssim 2^{-k}
\]
and
\[
\| N_{k} \|_{p} \lesssim_{p} 2^{2k} \|\tilde F_{0}\|_{2}^{2/p} \|\tilde F_{0}\|_{\infty}^{2-2/p},
\quad N_{k} := \sum_{T\in T_{k}} 1_{\vec I_{T}},
\quad 1\leq p<\infty.
\]
Choosing $p=p_{0}/2$ we obtain the bound
\[
\| N_{k} \|_{p} \lesssim_{p} 2^{2k}.
\]
For a fixed $k$ we will show
\[
|\Lambda^{\epsilon}_{\mathbf{P}_{k}}(\tilde F_{0},\tilde F_{1},\tilde F_{2})| \lesssim 2^{-\delta k}
\]
for some $\delta>0$, depending only on the $p_{i}$'s, to be determined later.

Let $\DI_{\eset}$ denote the collection of the maximal one-dimensional dyadic intervals of the form $\{x_{0}\}\times J_{1} \subset \eset_{2}$.
For each one-dimensional interval $J=\{x_{0}\}\times J_{1}\in\DI_{\eset}$ let
\[
\Omega_{J} := \{ \omega : |\omega||J|=1, \exists T\in\mathbf{T}_{k} : \vec I_{T}\supseteq J, \omega\supseteq \omega_{T}\}.
\]
Let
\[
G:= \sum_{J\in\DI_{\eset}} G_{J},
\quad
G_{J}(x_{0},x_{1}) := 1_{J}(x_{0},x_{1}) \sum_{\omega\in\Omega_{J}} (\Pi_{J_{1}\times\omega} \tilde F_{2}(x_{0},\cdot))(x_{1}).
\]
The sum defining the function $G$ is pointwise finite, and $G$ is measurable since $\tilde F_{2}$ is a dyadic test function.

We claim that for every $P = \vec I \times \omega_{1} \in\mathbf{P}_{k}$ we have
\[
\Lambda_{P}(\tilde F_{0},\tilde F_{1},\tilde F_{2})
=
\Lambda_{P}(\tilde F_{0},\tilde F_{1},G).
\]
Since $E_{2}\subset \eset_{2}$ by construction and the collection $\DI_{\eset}$ covers $\eset_{2}$, it suffices to show
\begin{multline*}
\int_{J_{1}} \tilde F_{1}(x_{0},x_{2}) \h_{I_{0}}(x_{0}) \tilde F_{2}(x_{0},x_{1}) w_{I_{1}^{j}\times \omega_{1}}(x_{1}) \dif x_{1}\\
=
\int_{J_{1}} \tilde F_{1}(x_{0},x_{2}) \h_{I_{0}}(x_{0}) G_{J}(x_{0},x_{1}) w_{I_{1}^{j}\times \omega_{1}}(x_{1}) \dif x_{1}
\end{multline*}
for every $J=\{x_{0}\}\times J_{1}\in\DI_{\eset}$, every $x_{2}\in I_{2}$, and every $j\in\{\pm 1\}$.
If $I_{0}\times I_{1}\cap J=\emptyset$, then both sides vanish identically.
Otherwise we must have $x_{0}\in I_{0}$.
If now $I_{1}\subseteq J_{1}$, then by construction $\tilde F_{1}$ vanishes on $\{x_{0}\} \times I_{2}$, so both sides again vanish identically.
On the other hand, if $J_{1}\subsetneq I_{1}$, then by construction $\Omega_{J}$ contains an ancestor of $\omega_{1}$, so the integrals coincide again.
This finishes the proof of the claim.

Now we estimate $\|G\|_{2}$.
By H\"older and Hausdorff--Young we get
\begin{align*}
\|G_{J}\|_{L^{2}(J)}^{2}
&=
\sum_{\omega\in\Omega_{J}} |\<\tilde F_{2}(x_{0},\cdot), w_{J_{1}\times\omega}\>|^{2}\\
&\leq
|\Omega_{J}|^{1-2/p_{1}'}(\sum_{\omega\in\Omega_{J}} |\<\tilde F_{2}(x_{0},\cdot), w_{J_{1}\times\omega}\>|^{p_{1}'})^{2/p_{1}'}\\
&\leq
|\Omega_{J}|^{1-2/p_{1}'} \| \tilde F_{2} \|_{L^{p_{2}}(J)}^{2} |J_{1}|^{1-2/p_{1}}.
\end{align*}
Maximality of $J\subset \eset_{2}$ gives an upper bound on the above $L^{p_{2}}(J)$ norm, and we obtain
\[
\|G_{J}\|_{L^{2}(J)}^{2}
\lesssim
|\Omega_{J}|^{1-2/p_{2}'} |J_{1}|
\leq
\int_{J} N_{k}^{1-2/p_{2}'}.
\]
Integrating these bounds and using monotonicity of $L^{p}$ norms (recall $|\eset_{2}|\lesssim 1$) we get
\begin{align*}
\|G\|_{2}^{2}
&\lesssim
\int_{\eset_{2}} N_{k}^{1-2/p_{2}'}
\lesssim
(\int_{\eset_{2}} N_{k}^{p} )^{(1-2/p_{2}')/p}\\
&\leq
\| N_{k} \|_{p}^{1-2/p_{2}'}
\lesssim
2^{2k(1-2/p_{2}')}.
\end{align*}
Normalize
\[
\tilde G:=2^{-k(1-2/p_{2}')} G,
\]
so that $\|\tilde G\|_{2}\lesssim 1$.
We claim
\[
|\Lambda^{\epsilon}_{\mathbf{P}_{k}}(\tilde F_{0},\tilde F_{1},\tilde G)| \lesssim 2^{-k}(1+pk),
\]
which would finish the proof.
By the tree selection algorithm in Proposition~\ref{prop:tree-selection} (beginning at some scale $l_{0}\leq 0$ with $\size^{(2)}(\mathbf{P}_{k},\tilde G) \leq 2^{-l_{0}}$) we partition
\[
\mathbf{P}_{k} = \cup_{l=l_{0}}^{\lceil pk \rceil} \cup_{T\in \mathbf{T}_{k,l}} T \cup \mathbf{P}_{k}',
\]
where
\[
\size^{(2)}(T,\tilde G) \lesssim 2^{-l},
\quad
\size^{(1)}(T,\tilde F_{1}) \lesssim \min(1,2^{-l})
\]
for $T\in \mathbf{T}_{k,l}$,
\[
\sum_{T\in \mathbf{T}_{k,l}} |\vec I_{T}| \lesssim 2^{2l},
\]
and
\[
\size^{(2)}(\mathbf{P}_{k}',\tilde G), \size^{(1)}(\mathbf{P}_{k}',\tilde F_{1}) \lesssim 2^{-pk}.
\]
By the single tree estimate \eqref{eq:single-tree-estimate:symmetric} we obtain
\[
|\sum_{l}\sum_{T\in\mathbf{T}_{k,l}} \Lambda^{\epsilon}_{T}(\tilde F_{0},\tilde F_{1}, \tilde G)|
\lesssim
\sum_{l=-\infty}^{\lceil pk \rceil} 2^{2l} 2^{-k} \min(1,2^{-l}) 2^{-l}
\lesssim
2^{-k}(1+pk).
\]
The remaining term can be written
\[
|\Lambda^{\epsilon}_{\mathbf{P}_{k}'}(\tilde F_{0},\tilde F_{1},\tilde G)|
=
|\sum_{T\in\mathbf{T}_{k}}\Lambda^{\epsilon}_{T\cap \mathbf{P}_{k}'}(\tilde F_{0},\tilde F_{1},\tilde G)|.
\]
Each $T\cap \mathbf{P}_{k}'$ is the disjoint union of a set of trees the union of whose top squares has measure bounded by $|\vec I_{T}|$.
We have
\[
\sum_{T\in\mathbf{T}_{k}} |\vec I_{T}|
\leq
\| \sum_{T\in\mathbf{T}_{k}} 1_{\vec I_{T}} \|_{p}^{p}
\lesssim
2^{2pk},
\]
so, again by the single tree estimate \eqref{eq:single-tree-estimate:symmetric},
\[
|\Lambda^{\epsilon}_{\mathbf{P}_{k}'}(\tilde F_{0},\tilde F_{1},\tilde G)|
\lesssim
2^{2pk} 2^{-k} 2^{-pk} 2^{-pk}
=
2^{-k},
\]
finishing the proof of the claim.
\end{proof}

\subsection{Triangles $b_{0}$ and $d_{10}$}
\begin{theorem}
\label{thm:main-restricted-type:F0}
Let $0< \alpha_{2} \leq 1/2 \leq \alpha_{0} < 1$ and $-1/2<\alpha_{1}<1/2$ satisfy \eqref{eq:Lp-range}.
Then for any measurable sets $E_{i} \subset A_{0}^{2}$, $i\in\{0,1,2\}$ there exists a major subset $E_{1}'\subset E_{1}$ (which can be taken equal to $E_{1}$ if $\alpha_{1}>0$)
such that for any dyadic test functions $|F_{i}| \leq 1_{E_{i}}$, $|F_{1}| \leq 1_{E_{1}'}$ with \eqref{eq:F0-diagonal}, $a\in A_{1}\setminus A_{0}$, or \eqref{eq:F0-fiberwise-character} we have
\[
|\Lambda^{\epsilon}(F_0,F_1,F_2)|
\lesssim_{\alpha_{0},\alpha_{1},\alpha_{2}}
\prod_{i=0}^{2} |E_{i}|^{\alpha_{i}},
\]
where the implied constant is independent of the choices of the scalars $|\epsilon_{\vec I}|\leq 1$ with $\epsilon_{\vec I}=0$ whenever $I_{i} \not\subset A_{0}$.
\end{theorem}
\begin{proof}
We can assume $|E_{0}|\leq 2^{-20} |E_{1}|$, since otherwise the conclusion follows from the local $L^{2}$ case with $E_{1}'=E_{1}$.

In case \eqref{eq:F0-fiberwise-character} we can also without loss of generality assume $E_{0} = A_{0} \times \tilde E_{0}$.
Setting $E_{1}' = E_{1} \setminus \tilde E_{0} \times A_{0}$ we get that the left-hand side of the conclusion vanishes identically.

In case \eqref{eq:F0-diagonal} we argue as in the proof of Theorem~\ref{thm:main-restricted-type} with the roles of indices $0$ and $2$ interchanged.
The main difference from the previous case is that the time-frequency projections in general need not be adapted to the good function $G$.
However, under the additional condition $a\in A_{1}\setminus A_{0}$ we may assume
\[
1_{B_{0}}(x_{1},x_{2}) = 1_{\tilde B_{0}}(x_{2} \wplus (a\wtimes x_{1})),
\]
and then the directional maximal function $\tilde M_{p_{0}}1_{B_{0}}$ coincides with the two-dimensional maximal function $M_{p_{0}}1_{B_{0}}$.
It follows that for every $J\in\DI_{\eset}$ and every bitile $P=\vec I\times\omega_{1}\in\mathbf{P}$ we have either $J\cap I_{1}\times I_{2} = \emptyset$ or $J_{1}\subsetneq I_{1}$, which in turn implies
\[
\Pi^{(0)}_{\mathbf{P}_{k}} \tilde F_{0} = \Pi^{(0)}_{\mathbf{P}_{k}} G.
\]
Thus we may replace $\tilde F_{0}$ by $G$ in the single tree estimates.
\end{proof}

\appendix
\section{Known special cases}
\label{appendix:dyadic}

Let us discuss briefly how our main result specializes to some cases that have already appeared in the literature in a very similar form.

\subsection{Maximally modulated Haar multiplier}
\label{appendix:dyadic:max-mod-haar}
Since the ordinary Haar multipliers
\[
(H^{\epsilon}f)(x) := \sum_{I} \epsilon_I |I|^{-1} \<f,\h_I\> \h_I(x),
\]
where $|\epsilon_I|\leq 1$ for each dyadic interval $I$, constitute a good dyadic model for the Hilbert transform, the maximally modulated Haar multipliers
\begin{equation}
\label{eq:maximal-def}
(H^{\epsilon}_{\star}f)(x) := \sup_{N} |(H^{\epsilon}M_N f)(x)|
\end{equation}
provide a reasonable algebraic model for the Carleson operator, albeit different from the model of truncated Walsh--Fourier series considered e.g.\ in \cite{MR0217510}.
Here $M_N$ simply represents the Walsh modulation operator,
\[
(M_N f)(x) := w_{N}(x) f(x).
\]

Let $\epsilon_{\vec I}=\epsilon_{(I_0,I_1,I_2)}$ depend only on the interval $I_0$ and take two functions $f$ and $g$ on $A_{0}$.
Suppose that $N\colon A_{0}\to\{0,1,2,\ldots\}$ is a choice function that linearizes the supremum in \eqref{eq:maximal-def}.
If we substitute
\begin{align*}
F_0(x_{1},x_{2}) &:= f(x_{1}\wplus x_{2}),\\
F_1(x_{2},x_{0}) &:= \mathop{\mathrm{sgn}}g(x_{0})\sqrt{|g(x_{0})|} \,w_{N(x_{0})}(x_{2}),\\
F_2(x_{0},x_{1}) &:= \sqrt{|g(x_{0})|} \,w_{N(x_{0})}(x_{1}\wplus x_{0})
\end{align*}
into \eqref{eq:THT-def2}, we will obtain for $\Lambda^{\epsilon}(F_0,F_1,F_2)$ the equal expression
\[
\sum_{\vec I\in\DI} \frac{\epsilon_{I_0}}{|I_0|} \iiint f(x_{1}\wplus x_{2}) g(x_{0}) w_{N(x_{0})}(x_{1}\wplus x_{0}) w_{N(x_{0})}(x_{2}) \h_{I_1}(x_{1}) \h_{I_2}(x_{2}) \h_{I_0}(x_{0}) \dif \vec x.
\]
Here and later in this appendix we use the convention $x_{i},y_{i}\in I_{i}$ for integration domains, unless specified otherwise.
By the character property of the Walsh functions and the fact that the Haar functions are simply restrictions of the Rademacher functions to the corresponding intervals this equals
\[
\sum_{\vec I\in\DI} \frac{\epsilon_{I_0}}{|I_0|} \iiint f(x_{1}\wplus x_{2}) g(x_{0}) w_{N(x_{0})}(x_{1}\wplus x_{2}\wplus x_{0}) r_{k}(x_{1}\wplus x_{2}\wplus x_{0}) \dif \vec x.
\]
By changing the variables $y_{0}=x_{1}\wplus x_{2}$ (for fixed $x_{1}$) and observing $y_{0}\in I_1\wplus I_2=I_0$, the above equals
\[
\sum_{\vec I\in\DI} \frac{\epsilon_{I_0}}{|I_0|} \iiint f(y_{0}) g(x_{0}) w_{N(x_{0})}(y_{0}\wplus x_{0}) r_{k}(y_{0}\wplus x_{0}) \dif y_{0} \dif x_{1} \dif x_{0}.
\]
Observe that at each scale $k$ the integral $\sum_{I\in\mathbf{I}_{k}}\int_{x_{1}\in I_{1}}$ can be disregarded as it simply integrates over the union of intervals $I_1$, which is $A_{0}$.
Using the character property once again we obtain
\begin{align*}
&\sum_{I_{0}} \frac{\epsilon_{I_0}}{|I_0|} \iint f(y_{0}) g(x_{0}) w_{N(x_{0})}(y_{0}) w_{N(x_{0})}(x_{0}) \h_{I_0}(y_{0}) \h_{I_0}(x_{0}) \dif y_{0} \dif x_{0}\\
&=
\int_{\W} \sum_{I_0} \frac{\epsilon_{I_0}}{|I_0|} \<w_{N(x_{0})}f,\h_{I_0}\> \h_{I_0}(x_{0}) w_{N(x_{0})}(x_{0}) g(x_{0}) \dif x_{0}\\
&=
\int (M_{N(x_{0})} H^{\epsilon} M_{N(x_{0})}f)(x_{0}) g(x_{0}) \dif x_{0}.
\end{align*}
From the established bound for $\Lambda^{\epsilon}$ in Theorem~\ref{thm:main} using duality we deduce
\[
\|H^{\epsilon}_{\star}f\|_{p}\lesssim\|f\|_p
\quad\text{for any}\quad
1<p<\infty.
\]

\subsection{Walsh model of uniform bilinear Hilbert transform}
Theorem~\ref{thm:main} implies a bound for the trilinear form
\[
\Lambda^{\epsilon,L}_{\mathrm{BHT}}(f,g,h)
:=
\int \sum_{k} \sum_{\substack{I\in\mathbf{I}_{k}\\ \omega:|\omega|=2^{-k}}}\! \epsilon_{I}
\big(\Pi_{I\times (\omega\wplus 2^{-k})}f\big) \big(\Pi_{I\times (2^{L}\omega)}g\big)
\big(\Pi_{I\times (2^{L}\omega \wplus \omega \wplus 2^{-k})}h\big),
\]
where $\epsilon=(\epsilon_I)_I$ is a sequence of coefficients indexed by dyadic intervals and satisfying $|\epsilon_I|\leq 1$, while $L$ is an arbitrary positive integer.
This observation is interesting because a single estimate for the triangular Hilbert transform implies bounds for a sequence of one-dimensional trilinear forms $\Lambda^{\epsilon,L}_{\mathrm{BHT}}$ with constants independent of $\epsilon$ and $L$.

This form is similar to, but different from the trilinear form studied in \cite{MR2997005}.
As in Section~\ref{appendix:dyadic:max-mod-haar}, the discrepancy is due to the fact that our model is based on the algebraic structure of the Walsh field rather than on the order structure.

In order to apply Theorem~\ref{thm:main} substitute
\begin{align*}
F_0(x_{1},x_{2}) &:= f(x_{1} \wplus x_{2} \wplus 2^{-L}x_{2}),\\
F_1(x_{2},x_{0}) &:= h(2^{-L}x_{2} \wplus x_{0}),\\
F_2(x_{0},x_{1}) &:= g(x_{0}\wplus 2^{-L}x_{0} \wplus 2^{-L}x_{1})
\end{align*}
into \eqref{eq:THT-def2} to obtain
\begin{align*}
\Lambda^{\epsilon}(F_0,F_1,F_2)
=\sum_{k} 2^{-k} \sum_{\vec I\in\DI_{k}} \epsilon_{\vec I}
\iiint f(x_{1}\wplus x_{2}\wplus 2^{-L}x_{2}) g(x_{0}\wplus 2^{-L}x_{0}\wplus 2^{-L}x_{1}) & \\[-2mm]
h(2^{-L}x_{2}\wplus x_{0}) r_k(x_{1}\wplus x_{2}\wplus x_{0}) \dif x_{1} \dif x_{2} \dif x_{0} & .
\end{align*}
Observe that $x_{i}\in I_i$, $i=0,1,2$, implies
\begin{align*}
& x_{1} \wplus x_{2} \wplus 2^{-L}x_{2}\in I_1\wplus I_2\wplus 2^{-L}I_2 = I_0\wplus 2^{-L}I_2,\\
& x_{0}\wplus 2^{-L}x_{0} \wplus 2^{-L}x_{1}\in I_0\wplus 2^{-L}(I_0\wplus I_1) = I_0\wplus 2^{-L}I_2,\\
& 2^{-L}x_{2} \wplus x_{0}\in I_0\wplus 2^{-L}I_2,
\end{align*}
so we should expand $f,g,h$ into the Walsh-Fourier series on the dyadic interval $I=I_0\wplus 2^{-L}I_2$ of length $2^k$, i.e.\ into the wave packets with fixed eccentricity:
\begin{align*}
f(x_{1}\wplus x_{2}\wplus 2^{-L}x_{2}) &= 2^{-k} \sum_{m_{0}=0}^{\infty} \<f,1_I w_{m_{0} 2^{-k}}\> w_{m_{0} 2^{-k}}(x_{1}\wplus x_{2}\wplus 2^{-L}x_{2}),\\
g(x_{0}\wplus 2^{-L}x_{0}\wplus 2^{-L}x_{1}) &= 2^{-k} \sum_{m_{2}=0}^{\infty} \<g,1_I w_{m_{2} 2^{-k}}\> w_{m_{2} 2^{-k}}(x_{0}\wplus 2^{-L}x_{0}\wplus 2^{-L}x_{1}),\\
h(2^{-L}x_{2}\wplus x_{0}) &= 2^{-k} \sum_{m_{1}=0}^{\infty} \<h,1_I w_{m_{1} 2^{-k}}\> w_{m_{1} 2^{-k}}(2^{-L}x_{2}\wplus x_{0}).
\end{align*}
Inserting these into the previous expression for $\Lambda^{\epsilon}(F_0,F_1,F_2)$ we obtain
\begin{align*}
\sum_{k} 2^{-4k} \sum_{\vec I\in\DI_{k}} \epsilon_{\vec I} \sum_{m_{0},m_{2},m_{1}}
& \<f,1_I w_{m_{0} 2^{-k}}\> \<g,1_I w_{m_{2} 2^{-k}}\> \<h,1_I w_{m_{1} 2^{-k}}\> \\[-2mm]
& \Big(\int_{I_1} e\big((m_{0}\wplus m_{2} 2^{-L}\wplus 1)2^{-k}\wtimes x_{1}\big) dx_{1}\Big)\\
& \Big(\int_{I_2} e\big((m_{0}\wplus m_{0} 2^{-L}\wplus m_{1} 2^{-L}\wplus 1)2^{-k}\wtimes x_{2}\big) dx_{2}\Big)\\
& \Big(\int_{I_0} e\big((m_{2}\wplus m_{2} 2^{-L}\wplus m_{1}\wplus 1)2^{-k}\wtimes x_{0}\big) dx_{0}\Big).
\end{align*}
Since we are integrating over intervals of length $2^k$, the above summands vanish unless
\[
m_{0}\wplus m_{2} 2^{-L}\wplus 1,\ m_{0}\wplus m_{0} 2^{-L}\wplus m_{1} 2^{-L}\wplus 1,\ \text{and}\ m_{2}\wplus m_{2} 2^{-L}\wplus m_{1}\wplus 1
\]
all belong to $A_{0}$, which is easily seen to be equivalent to the conditions
\[
m_{0}\wplus m_{2}\wplus m_{1}=0\quad \text{and}\quad m_{0}\wplus m_{2} 2^{-L}\wplus 1\in A_{0}.
\]
Moreover, in that case the three functions under the integrals over $I_1,I_2,I_0$ are precisely the constants
\[
2^{k} e\big((m_{0}\wplus m_{2} 2^{-L}\wplus 1)2^{-k}\wtimes l(I_i)\big),\ \ i=1,2,0,
\]
where $l(I_i)$ is the left endpoint of $I_i$.
Because of $0\in I_0\wplus I_1\wplus I_2$ they multiply to $2^{3k}$.
Allow the coefficients $\epsilon_{\vec I}$ to depend on $I=I_0\wplus 2^{-L}I_2$ only and observe that each interval $I\in\mathbf{I}_{k}$ appears for exactly $2^{-k}$ choices of $\vec I$ as they range over $\DI_{k}$.
(Indeed, $I_2$ is arbitrary and $I_0,I_1$ are then uniquely determined.)
We end up with
\[
\sum_{k} 2^{-2k} \sum_{I\in\mathbf{I}_{k}} \epsilon_{I} \sum_{\substack{m_{0},m_{2},m_{1}\\ m_{0}\wplus m_{2}\wplus m_{1}=0\\ m_{0}\wplus m_{2} 2^{-L}\wplus 1\in A_{0}}}
\<f,1_I w_{m_{0} 2^{-k}}\> \<g,1_I w_{m_{2} 2^{-k}}\> \<h,1_I w_{m_{1} 2^{-k}}\>,
\]
i.e., by substituting $m=m_{0}\wplus 1$ and $n=m_{2}\wplus(m_{0}\wplus 1)2^L$,
\begin{align}
\label{eq:BHTcase-expanded}
\sum_{k} 2^{-2k} \sum_{I\in\mathbf{I}_{k}} \epsilon_{I} \sum_{\substack{m,n\\ 0\leq n<2^L}}
\<f,1_I w_{(m\wplus 1)2^{-k}}\> \<g,1_I w_{(m 2^L\wplus n)2^{-k}}\> & \\[-3mm]
\<h,1_I w_{(m 2^L\wplus n\wplus m\wplus 1)2^{-k}}\> & . \nonumber
\end{align}

On the other hand, we can start from $\Lambda^{\epsilon,L}_{\mathrm{BHT}}$ and write the dyadic interval $\omega$ explicitly as $\omega=[m 2^{-k},(m+1)2^{-k})$.
The three time-frequency projections appearing in the definition can be expanded using vertical decompositions into tiles as:
\begin{align*}
\Pi_{I\times (\omega\wplus 2^{-k})}f &= 2^{-k}
\<f,1_I w_{(m\wplus 1)2^{-k}}\> 1_I w_{(m\wplus 1)2^{-k}},\\
\Pi_{I\times (2^{L}\omega)}g &= 2^{-k} \sum_{n=0}^{2^L-1}
\<g,1_I w_{(m 2^L\wplus n)2^{-k}}\> 1_I w_{(m 2^L\wplus n)2^{-k}},\\
\Pi_{I\times (2^{L}\omega \wplus \omega \wplus 2^{-k})}h &= 2^{-k} \sum_{n'=0}^{2^L-1}
\<h,1_I w_{(m 2^L\wplus n'\wplus m\wplus 1)2^{-k}}\> 1_I w_{(m 2^L\wplus n'\wplus m\wplus 1)2^{-k}}.
\end{align*}
Observe that the integral
\[
\int \big(\Pi_{I\times (\omega\wplus 2^{-k})}f\big) \big(\Pi_{I\times (2^{L}\omega)}g\big)
\big(\Pi_{I\times (2^{L}\omega \wplus \omega \wplus 2^{-k})}h\big)
\]
is equal to
\[
2^{-2k} \sum_{n=0}^{2^L-1} \<f,1_I w_{(m\wplus 1)2^{-k}}\> \<g,1_I w_{(m 2^L\wplus n)2^{-k}}\>
\<h,1_I w_{(m 2^L\wplus n\wplus m\wplus 1)2^{-k}}\>,
\]
since the terms with $n\neq n'$ disappear.
That way we arrive at \eqref{eq:BHTcase-expanded} once again, completing the proof of
$\Lambda^{\epsilon}(F_0,F_1,F_2)=\Lambda^{\epsilon,L}_{\mathrm{BHT}}(f,g,h)$.

\subsection{Endpoint counterexample}
The observation from the previous section is also useful to explain the failure of some estimates at the boundary of the Banach triangle.
By formally taking $L\to\infty$ we are motivated to substitute
\[
F_0(x_{1},x_{2}) := f(x_{1} \wplus x_{2}),\ 
F_1(x_{2},x_{0}) := h(x_{0}),\ 
F_2(x_{0},x_{1}) := g(x_{0}),
\]
in which case \eqref{eq:THT-def2} becomes
\begin{align*}
& \sum_{\vec I\in\DI} \epsilon_{I_0} |I_0|^{-1} \Big(\iint f(x_{1}\wplus x_{2}) \h_{I_1}(x_{1})\h_{I_2}(x_{2}) \dif x_{1} \dif x_{2}\Big)
\Big(\int g(x_{0})h(x_{0}) \h_{I_0}(x_{0}) \dif x_{0}\Big)\\
& = \sum_{I_0} \epsilon_{I_0} |I_0|^{-1} \<f,\h_{I_0}\> \<gh,\h_{I_0}\>
= \int f(x) H^{\epsilon}(gh)(x) \dif x.
\end{align*}
Since Haar multipliers are generally not bounded on $L^1$, we see that Estimate \eqref{eq:lp-estimate} cannot hold when $p_0=\infty$.

The positive results in this limiting case do not reveal the true structural complexity of $\Lambda^{\epsilon}$.
Indeed, when one of the functions depends on a single variable alone (such as $F_0(x,y)=x$), then the triangle ``breaks'' immediately.
No techniques from time-frequency analysis are required to bound such degenerate cases, even though they correspond both to the limiting case $a\to\infty$ and to the special case $N\equiv 0$ in Theorem~\ref{thm:main}.

\section{Real triangular Hilbert transform}
\label{appendix:real}
In this appendix we show the equivalence of \eqref{eq:THT-real-def1} and \eqref{eq:THT-real-def2} and indicate how to obtain the Carleson operator from \eqref{eq:THT-real-def2}.

\subsection{Equivalence of the definitions}
For $\vec\beta_{0},\vec\beta_{1},\vec\beta_{2}\in\R^{2}$ in general position consider the change of variables
\[
\begin{pmatrix}
t\\ u\\ v
\end{pmatrix}
=
B
\begin{pmatrix}
x_0\\ x_1\\ x_2
\end{pmatrix},
\quad
B :=
\begin{pmatrix}
1 & 1 & 1\\
\vec\beta_{0} & \vec\beta_{1} & \vec\beta_{2}
\end{pmatrix}.
\]
If $\pi_i$ denotes the projection $\pi_i\colon\R^3\to\R^2$, $\pi_i(x_0,x_1,x_2)=(x_{i+1},x_{i-1})$, then for arbitrary functions $F_0,F_1,F_2$ we have
\begin{align*}
\Lambda_{\Delta}(F_{0},F_{1},F_{2})
&=
\iiint \prod_{i=0}^{2}F_{i}(\pi_i(x_{0},x_{1},x_{2})) \frac{1}{x_{0}+x_{1}+x_{2}} \dif x_{0} \dif x_{1} \dif x_{2}\\
&=
|\det B|^{-1}
\iiint
\prod_{i=0}^{2} F_{i}(\pi_i B^{-1}(t,u,v))
\frac{\dif t}{t} \dif u \dif v\\
&=
|\det B|^{-1}
\iiint
\prod_{i=0}^{2} \tilde F_{i}((u,v)-\vec{\beta_{i}}t)
\frac{\dif t}{t} \dif u \dif v\\
&=
|\det B|^{-1} \Lambda_{\vec\beta_{0},\vec\beta_{1},\vec\beta_{2}}(\tilde F_{0}, \tilde F_{1}, \tilde F_{2}),
\end{align*}
where
\begin{equation}\label{eq:fn-substitution}
\tilde F_{i}(u,v) := F_{i}(\pi_i B^{-1}(0,u,v)).
\end{equation}
Here we have used the fact that
\[
\pi_i B^{-1}
\begin{pmatrix}
1 \\ \vec\beta_{i}
\end{pmatrix}
=
\begin{pmatrix}
0\\ 0
\end{pmatrix}.
\]
The surprising observation is now that
\[
\|\tilde F_{i}\|_{p_{i}} = |\det B|^{1/p_{i}} \|F_{i}\|_{p_{i}}.
\]
Indeed, the change of variables in the definition of $\tilde F_{i}$ is given by the $2\times 2$ submatrix of $B^{-1}$ obtained by crossing out the first column and the $i$-th row.
By Cramer's rule the determinant of that submatrix equals $\det B^{-1}$ times the $(1,i)$-th entry of $B$, up to the sign.
Since the latter entry of $B$ is $1$, the determinant of the change of variables is $\pm(\det B)^{-1}$.
The ratio of $L^{p_{i}}$ norms equals the absolute value of the determinant to the power $-1/p_{i}$, as required.

This shows that an $L^{p_{0}} \times L^{p_{1}} \times L^{p_{2}}$ estimate for $\Lambda_{\Delta}$ cannot be worse than the corresponding estimate for $\Lambda_{\vec\beta_{0},\vec\beta_{1},\vec\beta_{2}}$.
Running the above argument backwards we also obtain the converse.

\subsection{Less singular two-dimensional forms}
The trilinear forms introduced in \cite{MR2597511} can be written as
\[
\Lambda^{K}_{B_{0},B_{1},B_{2}}(F_{0},F_{1},F_{2})
:=
\iint_{\R^{2}} \mathrm{p.v.}\int_{\R^2} \prod_{i=0}^{2} F_i\big(\vec{x}-B_{i}\vec{t}\big) K(\vec{t}) \dif\vec{t} \dif\vec{x},
\]
where $B_1,B_2,B_3$ are now $2\times 2$ real matrices (interpreted as linear operators on $\R^2$) and $K$ is a two-dimensional Calder\'on--Zygmund kernel. If $K$ is odd and homogeneous of degree $-2$, then it takes the form
\[
K(r\cos\theta,r\sin\theta) = \frac{\Omega(\theta)}{r^2},\ \ \Omega(\theta+\pi)=-\Omega(\theta),\ \text{for}\ 0<r<\infty,\ \theta\in\R/(2\pi\Z).
\]
Observe that
\begin{align*}
& \mathrm{p.v.}\int_{\R^2} \prod_{i=0}^{2} F_i\big(\vec{x}-B_{i}\vec{t}\big) K(\vec{t}) \dif\vec{t}\\
& = \mathrm{p.v.}\int_{0}^{2\pi} \int_{0}^{\infty} \prod_{i=0}^{2} F_i\big(\vec{x}-B_{i}(r\cos\theta,r\sin\theta)\big) \frac{\Omega(\theta)}{r^2} r \dif r \dif \theta\\
& = \int_{0}^{\pi} \Omega(\theta) \,\mathrm{p.v.}\int_{\R} \prod_{i=0}^{2} F_i\big(\vec{x}-r B_{i}(\cos\theta,\sin\theta)\big) \frac{\dif r}{r} \dif\theta,
\end{align*}
so that
\[
\Lambda^{K}_{B_{0},B_{1},B_{2}} = \int_{0}^{\pi} \Omega(\theta) \,\Lambda_{B_{0}(\cos\theta,\sin\theta), B_{1}(\cos\theta,\sin\theta), B_{2}(\cos\theta,\sin\theta)} \dif\theta,
\]
i.e.\@ $\Lambda^{K}_{B_{0},B_{1},B_{2}}$ is a ``superposition'' of the forms \eqref{eq:THT-real-def1}.
Consequently, $L^p$ estimates for all cases of the matrices studied in \cite{MR2597511} and the remaining case from \cite{MR2990138} would follow from a single estimate for \eqref{eq:THT-real-def2}, even uniformly over all choices of $B_0,B_1,B_2$.

\subsection{Carleson maximal operator}
Analogously to the dyadic case but with an additional smooth cutoff we consider
\begin{align*}
F(x,y) &:= f(-x-y) \,D_{L}^{p}\phi(x),\\
G(y,z) &:= e_{N(z)}(y) \mathop{\mathrm{sgn}}g(z) \sqrt{|g(z)|} \,D_{L}^{2p'}\phi(y+z),\\
H(z,x) &:= e_{N(z)}(z+x) \sqrt{|g(z)|} \,D_{L}^{2p'} \phi(x)
\end{align*}
for $f\in L^{p}(\R)$ and $g\in L^{p'}(\R)$, $2<p<\infty$, where $e_{N}(x)=e^{2\pi i N x}$, $N$ is a measurable linearizing function for the Carleson operator, $\phi$ is a smooth positive function with compact support, and $D_{L}^{p}\phi(x) = L^{-1/p}\phi(x/L)$.
Then
\[
\Lambda_{\Delta}(F,G,H)
= \iiint f(-x-y)
e_{N(z)}(x+y+z) g(z)
 D_{L}^{p}\phi(x) D_{L}^{2p'}\phi(x) D_{L}^{2p'}\phi(y+z) \frac{\dif(x,y,z)}{x+y+z}.
\]
The change of variables $t=x+y+z$ gives
\[
\iiint f(z-t)
e_{N(z)}(t) g(z)
D_{L}^{p} \phi(x) D_{L}^{2p'} \phi(x) D_{L}^{2p'} \phi(t-x) \frac{\dif(x,z,t)}{t},
\]
which converges to a constant times
\[
\iint \frac{1}{t} f(z-t)
e_{N(z)}(t) g(z)
\dif(z,t)
\]
as $L\to\infty$ and yields an $L^p$ bound for
\[
(C_{N} f)(z) = \mathrm{p.v.}\int_{\mathbb{R}} f(z-t) e^{2\pi iN(z)t} \frac{\dif t}{t}.
\]

\subsection{Multilinear generalization}
For any positive integer $n$ one can also consider the straightforward $(n+1)$-linear generalization of \eqref{eq:THT-real-def1} given by
\begin{equation}
\label{eq:SHT-general-def}
\Lambda_{\vec\beta_{0},\vec\beta_{1},\ldots,\vec\beta_{n}}(F_{0},F_{1},\ldots,F_{n})
:=
\int_{\R^{n}} \mathrm{p.v.}\int_{\R} \prod_{j=0}^{n}F_j(\vec x-\vec\beta_{j}t) \frac{\dif t}{t} \dif \vec x,
\end{equation}
where this time $\vec\beta_{j}\in\R^{n}$ and the functions $F_j$ are $n$-dimensional. This object is expected to be even more difficult, as conjectured bounds for the one-dimensional trilinear Hilbert transform
\begin{equation}
\label{eq:TriHT-def}
\Lambda_{\mathrm{3HT}}(f_0,f_1,f_2,f_3) := \int_{\mathbb{R}} \mathrm{p.v.}\int_{\mathbb{R}} f_0(x) f_1(x-t) f_2(x-2t) f_3(x-3t) \frac{\dif t}{t} \dif x
\end{equation}
and its variants would follow from bounds for \eqref{eq:SHT-general-def} when $n=3$.
No positive results are known for this operator; see \cite{MR2449537} for some negative results.
However, an interesting observation is that the linearized polynomial Carleson operator
\[
(C_{N_1,\ldots,N_{n-1}} f)(x) = \mathrm{p.v.}\int_{\mathbb{R}} f(x-t) e^{i(N_1(x)t + N_2(x)t^2 + \cdots + N_{n-1}(x)t^{n-1})} \frac{\dif t}{t}
\]
can be encoded into \eqref{eq:SHT-general-def}.
Let $\vec\beta_{0}$ be the origin and let $\vec\beta_{1},\ldots,\vec\beta_{n}$ constitute the standard basis for $\R^n$.
The identity
\[
\sum_{j=0}^{k} (-1)^{k-j} \binom{k}{j} j^m = \left\{\begin{array}{cl} 0 & \text{ for } m=0,1,\ldots,k-1, \\ k! & \text{ for } m=k \end{array}\right.
\]
can be shown easily by induction on a positive integer $k$ and its immediate consequence is
\[
\sum_{j=0}^{n-1} \sum_{k=\max\{j,1\}}^{n-1} (-1)^{j}\frac{1}{k!}\binom{k}{j}N_k(x_n)\Big(\sum_{l=1}^{k}l x_l-jt\Big)^k = \sum_{k=1}^{n-1} N_k(x_n) t^k .
\]
It follows that with
\[
F_j(x_1,\ldots,x_n) = g_j(x_n) \prod_{k=\max\{j,1\}}^{n-1} \exp\bigg(i(-1)^{j}\frac{1}{k!}\binom{k}{j}N_k(x_n)\Big(\sum_{l=1}^{k}l x_l\Big)^k\bigg)
\]
for $j=0,1,\ldots,n-1$ and
\[
F_n(x_1,\ldots,x_n) = f(x_n),
\]
the form \eqref{eq:SHT-general-def} formally becomes $\int (C_{N_1,\ldots,N_{n-1}} f) g_0 g_1\cdots g_{n-1}$.
To be precise we should also include appropriate cutoffs of the functions $F_j$, similarly as we did in the previous subsection.
It would be interesting to investigate which particular cases of $\Lambda_{\vec\beta_{0},\vec\beta_{1},\ldots,\vec\beta_{n}}$ can be resolved using the techniques from \cite{MR2545246} and \cite{arXiv:1105.4504}.

\printbibliography
\end{document}